\documentclass[12pt,a4paper]{article}
\usepackage{amsmath}
\usepackage{amssymb}
\usepackage{amsthm}
\usepackage{amsfonts}
\usepackage{latexsym}
\usepackage{verbatim}

\theoremstyle{plain}
\newtheorem{thm}{Theorem}[section]

\newtheorem{lem}{Lemma}[section]

\theoremstyle{remark}

\theoremstyle{definition}
\newtheorem{defn}{Definition}[section]

\newtheorem{exmp}{Example}[section]

\title{Equivariant fixed point formulae and Toeplitz operators under Hamiltonian torus actions}
\author{Andrea Galasso \footnote{\noindent{\bf Address:} Room 407, Chee-Chun Leung Cosmology Hall, National Taiwan University; {\bf e-mail}: andrea.galasso@ncts.ntu.edu.tw}}
\date{}

\begin{document}
\maketitle

\begin{abstract} The main result of this paper is the description of asymptotics along rays in weight space of traces of equivariant Toeplitz operators composed with quantomorphisms for torus actions. The main ingredient in the proof is the microlocal analysis of the equivariant Szeg\"o kernels.  \end{abstract}

\section{Introduction}

Let $M$ be a $\mathrm{d}$-dimensional Hodge manifold with complex structure $J$ and Hodge form $\omega$. Then there exists a positive line bundle $A$ on $M$ with Hermitian structure $h$ and unique compatible connection $\nabla$ with curvature form $-2\pi\imath\,\omega$. We denote with $X$ the circle bundle lying in the dual line bundle $A^\vee$, with circle action $r\,:\, S^1\times X\rightarrow X$, projection $\pi:X\rightarrow M$ and $\partial_\theta$ the generator of the structure circle action on it. Hence $X$ is naturally a contact and Cauchy-Riemann manifold by positivity of $A$; if $\alpha$ is the contact form, $X$ inherits a volume form $\mathrm{dV}_X$. 

Suppose given, in addition, an action $\mu :G\times M\rightarrow M$ of a compact Lie group $G$, which is holomorphic with respect to the complex structure $J$ and Hamiltonian with moment map $\Phi:M\rightarrow \mathfrak{g}^\vee$, where $\mathfrak{g}$ denotes the Lie algebra of $G$. By \cite{k} the action $\mu$ naturally induces an infinitesimal contact action of $\mathfrak{g}$ on the circle bundle $X$. Explicitly, if $\xi\in \mathfrak{g}$ and $\xi_M$ is the corresponding Hamiltonian vector field on $M$, then its contact lift $\xi_X$ is as follows.
Let $\mathbf{v}^\sharp$ denote the horizontal lift on $X$ of a vector field $\mathbf{v}$ on $M$, we have
\begin{equation}
 \label{eqn:contact vector field}
\xi_X:= \xi_M^\sharp-\langle \Phi_G\circ \pi, \xi\rangle \,\partial_\theta.
\end{equation}

Furthermore, suppose that the infinitesimal action \eqref{eqn:contact vector field} can be integrated to an action $\tilde{\mu}:G\times X\rightarrow X$ of $G$. By hypothesis it preserves the Cauchy-Riemann structure and the contact form $\alpha$. There is a naturally induced unitary representation of $G$ on the Hardy space
$H(X)\subset L^2(X)$ given by
\[ g\,:\,f \mapsto f\circ \tilde{\mu}_{g^{-1}}\,. \] 
Thus $H(X)$ can be equivariantly decomposed over the irreducible representations of the group $G$:
\begin{equation}
 \label{eqn:decomposition of H(X)}
H(X)=\bigoplus _{\varpi\in \widehat{G}} H(X)_{\varpi},
\end{equation}
where $\widehat{G}$ is the collection of all irreducible representations of $G$; for each $\varpi\in \widehat{G}$ we denote with $\chi_{\varpi}$ the corresponding character. As is well-known, if $\Phi (m)\neq 0$ for every $m\in M$, then each isotypical component $H(X)_{\varpi}$
is finite dimensional (see e.g. \S 2 of \cite{pao-IJM}). The corresponding projector $$\Pi_{k\varpi}\,:\,L^2(X)\rightarrow H_{\varpi}(X)$$ is called the equivariant Szeg\"o projector. The aim of this paper is to apply equivariant Szeg\"o kernels $\Pi_{k\varpi}(x,\,y)$ to the asymptotic study of a class of trace formulae in equivariant geometric quantization and algebraic geometry in the setting of ``ladder representations'', see \cite{guillemin-sternberg hq}, \cite{pao-IJM}, \cite{pao-loa}, \cite{gp} and \cite{gp-asian}. 

Suppose we a have biholomorphism $\gamma_{M}\,:\,M\rightarrow M$ admitting a linearization to a unitary automorphism $\gamma_{A}$ of $(A,\,h)$. Then, $\gamma_M$ is a symplectomorphism and it induces a contactomorphism $\gamma_X\,:\, X \rightarrow X$. If $\gamma_M$ commutes with the action of $G$ then we have $\gamma_X\circ \tilde{\mu}_g =\tilde{\mu}_g\circ \gamma_X$ for every $g\in G$. Thus, we have an equivariant quantomoprhism
$${\gamma}_{k\varpi}\,:\,H_{k\varpi}(X)\rightarrow H_{k\varpi}(X)$$ 
which is a unitary automorphism of $H_{k\varpi}(X)$. In this article we will focus on the abelian case: $G$ is a $\mathrm{g}$-dimensional torus $\mathbb{T}$. Let us pause to give a simple example inspired from \cite{pao-IJM}. 

\begin{exmp} \label{exmp}
	Let us consider the unitary representation $\cdot\,:\,\mathbb{T}^1 \times \mathbb{C}^3 \rightarrow \mathbb{C}^3$ given by
	\[t \cdot (z_0,\, z_1,\, z_3) =: (t\, z_0,\, t^2\, z_1,\, t^3\, z_2)\,.\] 
	Consider a unitary diagonal matrix $\Gamma \in U(4)$. Both $\Gamma$ and $\cdot$ descend to actions on $\mathbb{P}^2$ with linearizations to the hyperplane line bundle $A$. The action commutes with the holomorphic symplectomorphism $\gamma$ induced by $\Gamma$. The moment map  $\Phi\,:\,\mathbb{P}^2\rightarrow \mathbb{R}$ associated to $\cdot$ is given by
	\[\Phi([z_0:z_1:z_2])=\frac{\lvert z_0 \rvert^2+2\,\lvert z_1 \rvert^2+3\,\lvert z_2 \rvert^2}{\lvert z_0 \rvert^2+\lvert z_1 \rvert^2+\lvert z_2 \rvert^2}\,. \]	
	
	The space $H(X)$ can be identified with the space of homogeneous polynomials of any degrees; $H_{\varpi}(X)$ is the subspace of those monomials $z_0^a\,z_1^b\,z_2^c$ such that $a+2b+3c=\varpi$. Notice that $a+b+c$ is not constant and thus $H_{\varpi}(X)$ is not contained in any space of homogeneous polynomials of given degree. Anyhow $\Gamma$ induces an action given by
	\[(\Gamma\cdot f)(\mathbf{z})=:f(\Gamma^{-1}\,\mathbf{z})\,, \]
	which preserves the isotypes $H_{\varpi}(X)$.
\end{exmp}

Every smooth function $f = f (m)$ on $M$ lifts to a $r$-invariant function $f = f (x)$ on $X$. We are led to the following definition, which is the main object of study of this paper.

\begin{defn} \label{def:psi} For any $\varpi \in \widehat{G}=\mathbb{Z}^g$, $k\in \mathbb{N}$  and smooth function $f$ we define equivariant Toeplitz operators
	\[T^{k\varpi}_{f}\,=\,\Pi_{k\varpi}\circ M_f \circ \Pi_{k\varpi}\,:\, H_{k\varpi}(X) \rightarrow H_{k\varpi}(X) \]
where $M_f$ denotes the multiplication operator. More generally we should consider the composition
\[\Psi_{k\varpi}=\Psi_{k\varpi}(\gamma,\,f)=: \tilde{\gamma}_{k\varpi}\circ T_f^{(k\varpi)}\,:\,H_{k\varpi}(X) \rightarrow H_{k\varpi}(X)\,. \]
\end{defn}

For each $\varpi \in \mathbb{Z}^g$, we shall prove that the trace of $\Psi_{k\varpi}$ admits a complete asymptotic expansion in decreasing power of $k$ and we shall give an explicit expression of the leading term. For the circle action case the study of Toeplitz operators in this setting was already developed in \cite{pao-loa}. 

Before giving the statement of the main theorem, some notations are
needed. Let us denote with $M_{\varpi}$ the pre-image via $\Phi$ of the ray $\mathbb{R}_+\cdot \varpi$. Under transversality assumption $M_{\varpi}$ (if non-empty) is a connected compact submanifold of $M$, of real codimension $\mathrm{g}-1$ (see Lemma $2.3$ in \cite{pao-IJM}). Fix $m\in M_{\varpi}$, the tangent space to the leaf of null foliation through $m$ is given by
\[ \mathrm{val}_m\left(\ker(\Phi(m)) \right)= \mathrm{val}_m\left(\mathrm{diag}(\varpi)^{\perp_{\mathfrak{t}}}\right)\,, \]
where $\mathrm{val}_m$ is the linear map $\xi \rightarrow \xi_M(m)$ induced by the valuation and $\perp_{\mathfrak{t}}$ denotes the orthogonal complement with respect to the standard scalar product $\langle\cdot,\,\cdot \rangle_{\mathfrak{t}}$ on $\mathfrak{t}$. Let us denote with $H=\exp_{\mathbb{T}}\left( \ker(\Phi(m)) \right)$. Thus, the space $\overline{M}_{\varpi}=: M_{\varpi}/H$ is an orbifold of complex dimension $\mathrm{e}=\mathrm{d}-\mathrm{g}+1$; we denote with $p_{\overline{M}_{\varpi}}\,:\, M_{\varpi}\rightarrow \overline{M}_{\varpi}$ the corresponding projection. Indeed, by the Slice Theorem a neighborhood of any orbit $H\cdot m=m_0$ is equivariantly diffeomorphic to a neighborhood of the zero section of the associated principal bundle $$H\times_{H_m} N_m\,,$$ where $N_m$ is the conormal space to $H\cdot m$ in $M_{\varpi}$ and $H_m$ is the stabilizer of $m$ for the action of $H$, which is finite dimensional (see Lemma $2.7$ in \cite{pao-IJM}). Therefore, for some $\epsilon>0$ and $B_{2\mathrm{e}}(\epsilon)\subseteq N_m$, one has a homeomorphism $B_{2\mathrm{e}}(\epsilon)/H_m \cong \overline{U}$ onto some neighborhood of $m_0$ in $\overline{M}_{\varpi}$. 

Given the weight $\varpi \in \mathbb{Z}^{g}\subseteq \mathfrak{t}$, we can define the circle group $S$ whose action $\nu\,:\,S\times M \rightarrow M$ on $M$ (respectively $\tilde{\nu}\,:\,S\times X \rightarrow X$ on $X$) is given by restriction of the action of $\mathbb{T}$. Let us denote with $\mathbb{T}_m$ the stabilizer of a point $m\in M$. The cardinality $\lvert\mathbb{T}_m\rvert$ need not be constant on $M$, but it does attain a generic minimal value $\lvert\mathbf{T}\rvert$ on some dense open subset $M'$, where $\mathbf{T} \subseteq \mathbb{T}$ is the stabilizer of each point in $M'$ (Corollary $B.47$ of \cite{ggk}). As a consequence $H_m$ equals a fixed value $\mathbf{H}$ on $M'$. 

Let us denote with $\overline{X}_{\varpi}=: X_{\varpi}/\mathbb{T}$ and $p_{\overline{X}_{\varpi}}\,:\,X_{\varpi}\rightarrow \overline{X}_{\varpi}$ the projection; we have the following diagram
\begin{align} 
&X_{\varpi}'  \,\xrightarrow{\pi} \,M_{\varpi }' \nonumber\\
p_{\overline{X}_{\varpi}}&\downarrow   \qquad\quad \downarrow \,p_{\overline{M}_{\varpi}} \label{eq: diagram} \\
&\overline{X}_{\varpi }'  \qquad\overline{M}_{\varpi}' \nonumber
\end{align}
Under the previous assumptions $\gamma_X$ descends naturally to an automorphism $\gamma_{\overline{X}_{\varpi}}\,:\,\overline{X}_{\varpi} \rightarrow \overline{X}_{\varpi}$ whose fixed point locus has $\ell\in \mathbb{N}$ connected components denoted by $F_{\overline{X}_{\varpi},1},\,\dots,\,F_{\overline{X}_{\varpi},\ell}$, each of which has complex dimension $\mathrm{d}_l$ and complex co-dimension $\mathrm{c}_{l}=: \mathrm{e}-\mathrm{d}_l$. Each component $F_{\overline{X}_{\varpi},1}$ pull-backs to the $r\times \nu$-invariant locus $\widetilde{F}_{\overline{X}_{\varpi},l}=:p^{-1}_{\overline{X}_{\varpi}}(F_{\overline{X}_{\varpi},\,l})$, which descends naturally to the locus $F_{\overline{M}_{\varpi},l}$ in $\overline{M}_{\varpi}$ thanks to diagram \eqref{eq: diagram}. Furthermore, given $x\in \widetilde{F}_{\overline{X}_{\varpi},l}$, there exists $\kappa_{j}$ for each $j=1,\dots, \lvert\mathbf{T}\rvert$ such that $\gamma_X(x)=\tilde{\mu}_{\kappa_j}(x)$. The last piece of notations consists in giving the following two definitions.

\begin{defn} \label{def:cl} Let $N_{{\overline{M}_{\varpi}},\,l}(m_0)$ be the normal space to $F_{\overline{M}_{\varpi},\,l}$ at $m_0$ in $\overline{M}_{\varpi}$, and let $\gamma_{N_l,\,m_0}\,:\,N_{{\overline{M}_{\varpi}},\,l}(m_0)\rightarrow N_{{\overline{M}_{\varpi}},\,l}(m_0)$ be te unitary map induced by the differential of $\gamma_{\overline{M}_{\varpi}}$. For each $t_j\in \mathbf{T}$ the map $\mathrm{d}_m\mu_{t_j}$ descends naturally to a map on $N_{{\overline{M}_{\varpi}},\,l}(m_0)$. For each $t\in \mathbf{T}$ we define the complex number
	\[c_l^{(j)}=: \mathrm{det}_{\mathbb{C}}\left(\mathrm{id}_{N_{l,\,m_0}}-\mathrm{d}_{m_0}\gamma_{\overline{M}_{\varpi}}^{-1}\circ \mathrm{d}_{m_0}\mu_{t_j} \right),\, \]
notice that it is non-zero and it is constant on $F_{{\overline{M}_{\varpi}},\,l}$.
\end{defn}

\begin{defn}
	Let us define:
	\begin{align*}
	\overline{f}(m)=\int_{\mathbb{T}} f(\mu_t(m)) \mathrm{dV}_{\mathbb{T}}(t)
	\end{align*}
\end{defn}

\begin{thm}\label{thm:main} Suppose that $f$ is a smooth function on $M$, $\gamma\,:\, M\rightarrow M$ given as above and $\boldsymbol{0}\notin \Phi(M)$. For each $\varpi \in \widehat{G}$ let $\Psi_{k\varpi}$ as in Definition \ref{def:psi}. 
	\begin{itemize}
		\item If $M_{\varpi}= \emptyset$, then $\Psi_{k\varpi}=O(k^{-\infty})$ for $k\gg 0$.
		\item If  $M_{\varpi}\neq \emptyset$ assume that $\Phi$ is transverse to the ray through $\varpi$. Then as $k\rightarrow +\infty$ there is an asymptotic expansion \begin{align*}
\mathrm{trace}\left(\Psi_{k\varpi}\right) \sim &\, \sum_{l=1}^{\ell}\left(\lVert \varpi \rVert\,\frac{k}{\pi}  \right)^{\mathrm{d}+1-\mathrm{g}-\mathrm{c}_l} \, \sum_{j=1}^{ \lvert\mathbf{T}\vert} \frac{\chi_{\varpi}(\kappa_j)^k}{c_l^{(j)}(\gamma)}\, \\
&\cdot\frac{1}{\lvert\mathbf{H}\rvert}\,\int_{F_{\overline{M}_{\varpi},l}} {\lVert \Phi(m) \rVert}  ^{-(\mathrm{d}+2-\mathrm{g}-\mathrm{c}_l)}\cdot \overline{f}(m)\, \mathrm{dV}_{F_{\overline{M}_{\varpi},l}}(m) \\ &\cdot \left(1+\sum_{a\geq 1}k^{-a/2}\,R_{a,l,\varpi}\right)\,. \end{align*}
	\end{itemize} 
\end{thm}

Let us briefly describe how this asymptotics are related to prior literature. First, If $\tilde{\gamma}$ is the identity and $f=1$, the trace of $\Psi_{k\varpi}$ computes $\dim (H_{k\varpi}(X))$ as $k$ goes to infinity. Of course these dimensions was already obtained in \cite{pao-IJM}. For general $f$ we are considering the asymptotics of the trace of the level-$k$ equivariant Toeplitz operator. For the circle action, these operators were studied in \cite{pao-loa} where formulae for compositions and commutators were proved. From the other side, if the action is the standard circle action and $f=1$ we provide a simple approaches for the asymptotics of the Lefschetz fixed point formula. 

The main tool in the proof is the local asymptotics for the projectors $\Pi_{k\varpi}$ as $k \rightarrow +\infty$ for torus actions contained in \cite{pao-IJM}. The proof of Theorem $4$ in \cite{pao-IJM} is based on techniques from micro-local analysis and it relies on asymptotics expansions for Szeg\"o kernels, see \cite{boutet-guillemin}, \cite{bs}, \cite{sz} and \cite{z}.

Finally let us remark that, the approach in proof of Theorem \ref{thm:main} is similar to \cite{pao-fpf} even if the perspective is different in the present paper. In fact, in \cite{pao-fpf} projectors appearing in Definition \ref{def:psi} are replaced by
\[\Pi_{k,\,\varpi}\,:\, L^2(X)\rightarrow  H_{k,\,\varpi}(X) \]
where $H_{k,\,\varpi}(X):=H_k(X)\cap H_{\varpi}(X)$ and $H_k(X)$ denote the isotypes of the standard circle action and they can be identified with $H^0(M,\,A^{\otimes k})$. Hence, the analysis in \cite{pao-fpf} fits with the general theme of Toeplitz quantization, see \cite{charles}, \cite{mm}, \cite{schlichenmaier} and \cite{mz} for the equivariant setting. Instead in our case $H_{k\varpi}(X)$ need not be contained in a space of global sections see Example \ref{exmp}. 

Our analysis builds on microlocal techniques that can be also applied in the almost complex symplectic setting, see \cite{sz}. For the sake of simplicity, we have restricted our discussion to the complex projective setting.

\bigskip

\textbf{Acknowledgments} I gratefully thank Prof. R. Paoletti for suggesting me the study of this problem.

\section{Proof of Theorem \ref{thm:main}}

Let $\mathrm{dV}_{X\times X}$ and $\mathrm{dV}_X$ denote the volume densities of $X\times X$ and $X$, respectively. Then,
\begin{align} \label{eq:tracedef}
\mathrm{trace}\left(\Psi_{k\varpi} \right)&=\int_{X\times X} \Pi_{k\varpi}\left(\gamma^{-1}_X(x),\,y\right)\,f(y)\, \Pi_{k\varpi}\left(y,\,x\right)\,\mathrm{dV}_{X\times Y}(x,\,y) \\
&= \int_{X\times X} \Pi_{k\varpi}\left(x,\,\gamma_X(y)\right)\,f(y)\, \Pi_{k\varpi}\left(y,\,x\right)\,\mathrm{dV}_{X\times Y}(x,\,y)\, \nonumber \\
&=\int_X \Pi_{k\varpi} \left(\gamma^{-1}_X(y),\,y\right)\,f(y)\,\mathrm{dV}_{X}(y)\,. \nonumber
\end{align}

Let us set
\[\mathfrak{m}(x,\,y)=:\mathrm{max}\left\{\mathrm{dist}_X\left(\mathbb{T}\cdot x,\,X_{\varpi}\right),\,\mathrm{dist}_X\left(\mathbb{T}\cdot x,\,\mathbb{T}\cdot y\right)\right\}\,.\]
The first step in the proof consists in recalling the following Theorem of \cite{pao-IJM}.

\begin{lem}[Theorem 3 of \cite{pao-IJM}] \label{lem:rapdec}
	Assume that $\boldsymbol{0}\notin \Phi(M)$ and $\Phi$ is transversal to $\mathbb{R}_+\cdot \varpi$. Let $C,\,\epsilon >0$. Then, uniformly for 
	\[\mathfrak{m}(x,\,y)\geq C\, k^{\epsilon-1/2} \,, \]
	we have ${\Pi}_{k\varpi}(x,\,y)=O(k^{-\infty})$ as $k\rightarrow +\infty$.
\end{lem} 

We can introduce a bump function $\varrho\,:\,[0,\,+\infty)\rightarrow \mathbb{R}$ with compact support such that $\geq 0$, supported in $[0,\,2\,C]$ and $\varrho\equiv 1$  on $[0,\,C]$. Hence, by Lemma \ref{lem:rapdec}, we have
\begin{align} \label{eq:tracevarrho}
\mathrm{trace}\left(\Psi_{k\varpi} \right)=\int_X \varrho\left(k^{1/2-\epsilon}\,\mathfrak{m}(\gamma^{-1}_X(x),\,x)\right)\Pi_{k\varpi} \left(\gamma^{-1}_X(x),\,x\right)\,f(x)\,\mathrm{dV}_{X}(x)\,. 
\end{align}

We now refine once more the set of integration. Let us recall that $\mathrm{Fix}(\gamma_{\overline{X}_{\varpi}})$  is the fixed locus of $\gamma_{\overline{X}_{\varpi}}\,:\,\overline{X}_{\varpi} \rightarrow \overline{X}_{\varpi}$ and $p_{\overline{X}_{\varpi}}\,:\, X_{\varpi} \rightarrow \overline{X}_{\varpi}$ the projection. We will prove the following lemma.

\begin{lem} \label{lem:TmFix}
	There exists $C>0$ such that
	\[\mathrm{dist}_X\left(\mathbb{T}\cdot x,\,p^{-1}_{X,\,\varpi}({\mathrm{Fix}({\gamma}_{\overline{X}_{\varpi}})}) \right) \leq C\,k^{\epsilon-1/2},\,   \]
	for all $x$ in the support of \eqref{eq:tracevarrho}.
\end{lem}
\begin{proof}[Proof of Lemma \ref{lem:TmFix}]
	If $x$ lies in the support of integral \eqref{eq:tracevarrho}, then we have that $\mathfrak{m}(x,\,\gamma_X^{-1}(x))< C\,k^{\epsilon-1/2}$. Let $q\in X_{\varpi}$ be such that \begin{equation}\label{eq:y}\mathrm{dist}_X(x,\,X_{\varpi})=\mathrm{dist}_X(x,\,q)\,.\end{equation}
	Since the Riemannian metric $g_X$ on $X$ is $\mathbb{T}$-invariant, then $\overline{X}_{\varpi}$ is a Riemannian orbifold and for each $z_1,\,z_2 \in X_{\varpi}$ we have
	\begin{align}\mathrm{dist}_X(\mathbb{T}\cdot z_1,\,\mathbb{T}\cdot z_2)= \mathrm{dist}_{\overline{X}_{\varpi}}( p_{X,\,\varpi}(z_1),\, p_{X,\,\varpi}(z_2))\,, \label{eq:ineqGp}\end{align}
	where $\mathrm{dist}_{\overline{X}_{\varpi}}$ denote the Riemannian distance on ${\overline{X}_{\varpi}}$.
	
	By hypothesis, we have
	\begin{align*} 
	\mathrm{dist}_X\left( x,\,p^{-1}_{X,\,\varpi}{(\mathrm{Fix}({\gamma}_{\overline{X}_{\varpi}}))} \right) &\leq \mathrm{dist}_X\left( x,\,q \right)+\mathrm{dist}_X\left( q,\,p^{-1}_{X,\,\varpi}({\mathrm{Fix}({\gamma}_{\overline{X}_{\varpi}})}) \right) \\
	&\leq C\,k^{\epsilon-1/2}+\mathrm{dist}_{\overline{X}_{\varpi}}\left( p_{X,\,\varpi}(q),\,\mathrm{Fix}({\gamma}_{\overline{X}_{\varpi}}) \right)\,. \notag
	\end{align*}
	
	Thus, the statement follows by the following lemma.
	\begin{lem} \label{lem:distquotient} Let $q$ as in \eqref{eq:y}, then we have
	\[\mathrm{dist}_{\overline{X}_{\varpi}}\left( p_{X,\,\varpi}(q),\,\mathrm{Fix}({\gamma}_{\overline{X}_{\varpi}}) \right)\leq C\,k^{\epsilon-1/2}\,.\]
	\end{lem}
\begin{proof}[Proof of Lemma \ref{lem:distquotient}] The Lemma is a consequence of the inequalities
	\begin{equation}\label{eq:ineq1}\mathrm{dist}_{\overline{X}_{\varpi}}\left( p_{X,\,\varpi}(q),\,{\gamma}_{\varpi}(p_{X,\,\varpi}(q)) \right) \leq  C\,k^{\epsilon-1/2} \, \end{equation}
 	and
	\begin{equation}\label{eq:ineq2}\mathrm{dist}_{{\overline{X}_{\varpi}}}\left( p_{X,\,\varpi}(q),\,{\gamma}_{\varpi}(p_{X,\,\varpi}(q)) \right) \geq  \mathrm{dist}_{{\overline{X}_{\varpi}}}\left( p_{X,\,\varpi}(q),\,\mathrm{Fix}({\gamma}_{\overline{X}_{\varpi}}) \right)\,.  \end{equation}
	
	First, let us prove inequality \eqref{eq:ineq1}. By \eqref{eq:ineqGp} and recalling that $\gamma_X\circ p_{X,\,\varpi}= p_{X,\,\varpi} \circ \gamma_{\varpi}$, we have
	\begin{align*} &\mathrm{dist}_{\overline{X}_{\varpi}}\left( p_{X,\,\varpi}(q),\,{\gamma}_{\varpi}(p_{X,\,\varpi}(q)) \right) = 
	\mathrm{dist}_{X}\left(\mathbb{T}\cdot q,\,\mathbb{T}\cdot{\gamma}_{X}(q) \right) \\
	&\leq \mathrm{dist}_{X}\left(\mathbb{T}\cdot q,\,\mathbb{T} \cdot x \right) + 	\mathrm{dist}_{X}\left(\mathbb{T}\cdot x,\,\mathbb{T}\cdot{\gamma}_{X}(y) \right) \\
	&\leq C\,k^{\epsilon-1/2} +\mathrm{dist}_{X}\left(\mathbb{T}\cdot x,\,\mathbb{T} \cdot \gamma_X^{-1}(x) \right) + 	\mathrm{dist}_{X}\left(\mathbb{T}\cdot \gamma_X^{-1}(x),\,\mathbb{T}\cdot{\gamma}_{X}(q) \right) \\
	&\leq C\,k^{\epsilon-1/2}\,.
	\end{align*}
	with abuse the constant $C$ may change lines by lines.
	
	Now we show that inequality \eqref{eq:ineq2} holds true. Let us denote with  $F_{\overline{X}_{\varpi},\,l}$, $l=1,\dots \ell$, the connected components of $\mathrm{Fix}({\gamma}_{\overline{X}_{\varpi}})$ and with $N_{\overline{X}_{\varpi},l}$ the corresponding normal bundles. For each $\epsilon >0$, we define
	$$N^{(\epsilon)}_{\overline{X}_{\varpi},l}:=\{(x_0,\,n)\in N_{\overline{X}_{\varpi},l}\,:\, \lVert n \rVert < \epsilon \}$$
	and $F^{(\epsilon)}_{\overline{X}_{\varpi},\,l}\subseteq X$ is a tubular neighborhood of $F_{\overline{X}_{\varpi},\,l}$ such that
	\[ \exp_{l}\,:\,N^{(\epsilon)}_{\overline{X}_{\varpi},l} \rightarrow \overline{X}_{\varpi} \]
	induces a diffeomorphism onto $F^{(\epsilon)}_{\overline{X}_{\varpi},\,l}$ and furthermore $\overline{F_{\overline{X}_{\varpi},\,l_1}^{(\epsilon)}}\cap \overline{F_{\overline{X}_{\varpi},\,l_2}^{(\epsilon)}} = \emptyset$ for all $l_1\neq l_2\in\{1\,\dots, \ell \}$. Since $F_{\overline{X}_{\varpi},\,l}$ is compact and $\exp_l(x_0,\,\cdot)$ is an isometric immersion, we have
	\begin{equation} \label{eq:distineq}
	2\,\lVert n-n' \rVert \geq \mathrm{dist}_{\overline{X}_{\varpi}}\left(\exp_l\left(x_0,\,n\right),\, \exp_l\left(x_0,\,n'\right)\right) \geq \frac{1}{2}\,\lVert n-n' \rVert 
	\end{equation}
	for every $x_0\in F_{\overline{X}_{\varpi},\,l}$ and $(x_0,\,n),\,(x_0,\,n')\in N_{\overline{X}_{\varpi},l}^{(\epsilon')}\cap N_{\overline{X}_{\varpi},l}^{(\epsilon)}$.

	If $k$ is sufficiently large, on the support of \eqref{eq:tracevarrho}, we have $$p_{\overline{X}_{\varpi}}(q)\in \bigcup_{l=1}^{\ell}F_{\overline{X}_{\varpi},\,l}^{(\epsilon)}$$ and hence we can write $p_{\overline{X}_{\varpi}}(q)=\exp_l(q_0,n)$ for some $(q_0,\,n)\in N^{(\epsilon)}_{\overline{X}_{\varpi},l}$. Hence, making use of \eqref{eq:distineq}, we finally have
	\begin{align*}
	\mathrm{dist}_{\overline{X}_{\varpi}}\left( p_{\overline{X}_{\varpi}}(q),{\gamma}_{X}(p_{\overline{X}_{\varpi}}(q)) \right)	
	& =\mathrm{dist}_{\overline{X}_{\varpi}}\left(\exp_{\bar{X}_{\varpi}}\left(q_0,n\right),\gamma_{\overline{X}_{\varpi}}\circ \exp_{\bar{X}_{\varpi}}\left(q_0,n\right)\right) \\
	&= \mathrm{dist}_{\overline{X}_{\varpi}}\left(\exp_{\overline{X}_{\varpi}}\left(q_0,n\right), \exp_{\overline{X}_{\varpi}}\circ \,\mathrm{d}_{q_0}\gamma_{\overline{X}_{\varpi}}\left(n\right)\right) \\
	&\geq \frac{1}{2}\,\lVert \mathrm{d}_{q_0}\gamma_{\varpi}\left(n\right)-n \rVert \\
	& \geq \frac{1}{2}\,\inf\{\,\lvert \lambda_{i}-1\rvert\, \}\,\lVert n \rVert \\
	& \geq \frac{1}{4}\,\inf\{\,\lvert \lambda_{i}-1\rvert\, \}\,\mathrm{dist}_X\left(q,\,\mathrm{Fix}(\gamma_{\overline{X}_{\varpi}}) \right)
	\end{align*}
	where $\lambda_i$'s are the eigenvalues of $\mathrm{d}_{y_0}\gamma_{\overline{X}_{\varpi}}$ and $\lambda_i\neq 1$ for all $i$. Thus inequality \eqref{eq:ineq2} holds true and Lemma \ref{lem:distquotient} is proved.
\end{proof}
	The proof of Lemma \ref{lem:TmFix} is thus concluded. 
\end{proof}

Heisenberg coordinates for the unit circle bundle $X$ are defined in \cite{sz}, we recall briefly the notations and we refer to \cite{sz} for the precise definitions. Given $U \subseteq M$, a \textit{preferred local chart} $\mathfrak{F}\,:\,U\times B_{2\,\mathrm{d}}(\epsilon) \rightarrow \mathbb{C}^d$ is determined by a preferred local coordinates
\[\mathfrak{f}^{(m)}=:\mathfrak{F}(m,\,\cdot)\,:\,B_{2\,\mathrm{d}}(\epsilon) \rightarrow \mathbb{C}^d\,, \]
which means that it trivializes the unitary structure. On the other hand, a \textit{preferred local frame} $e\,:\,U\rightarrow A$ at $m$ is a local frame satisfying $e^*(m)=x$ and $\langle e^*,\, e \rangle=1$. Explicitly, for any $y \in X$ there exist a neighborhood $X(U)=\pi^{-1}(U)$ of $y\in X$ (where $U\subseteq M$), $\epsilon >0$ and a smooth map
\[ \Psi\, :\, X(U) \times B_{2\,\mathrm{d}}(\epsilon) \times (−\pi, \pi) \rightarrow X\]
such that 
\begin{align} \label{eq:heisenberg}
\psi^{(x)}\, =:\,
\Psi(x,\,\cdot,\,\cdot)\, :\, B_{2\,\mathrm{d}}(\epsilon) \times (−\pi, \pi) \rightarrow X,\qquad (\mathbf{z},\,\vartheta)\mapsto e^{\imath\,\vartheta}\cdot \frac{e^*\left(\mathfrak{f}(\mathbf{z})\right)}{\lVert e^*\left(\mathfrak{f}(\mathbf{z})\right) \rVert}\end{align}
which is called a \textit{Heisenberg local chart} for $X$ centered at $x$. Following \cite{sz}, if $\mathbf{w}\in T_mM$ and $\lVert \mathbf{w}\rVert< \epsilon$, we set
\[\psi(\mathbf{w},\,0)=:x+\mathbf{w}\,. \]

First, let us modify the moving preferred local chart on $M$. Let us restrict to $M'$, the locus where the stabilizer of each point is constantly equal to a fixed finite subset $\mathbf{T}$. Hence for each $m\in M'$ we have $H_m = \mathbf{H}$ and denote with
\[ p_{\overline{M}_{\varpi}}'\,:\, M_{\varpi}' \rightarrow \overline{M}_{\varpi}\,. \]
the projection. For each sufficiently small open subset $\overline{U}\subseteq \overline{M}_{\varpi}'$ let us set section $\sigma\,:\,\overline{U} \rightarrow M'$ for the projection $p_{\overline{M}_{\varpi}}'$. Since the action is locally free on $M'$, we have the following. 

\begin{lem} \label{lem:musigma}
The map $\mu_{\sigma}\,:\, H\times \overline{U} \rightarrow M_{\varpi}'$ given by 
\[\mu_{\sigma}(h,\,m_0)=:\mu_h(\sigma(m_0)) \] 
is a $\lvert \mathbf{H} \rvert$-to-$1$ covering; its image is a $\mu$-invariant tubular neighborhood of the $H$-orbit of $m=\sigma(m_0)$. 
\end{lem}

Fix a point ${m}_0\in \overline{M}_{\varpi}$ and consider a small neighborhood $\overline{U}$ of $m_0$ in $\overline{M}_{\varpi}$. By the discussion above, we have a covering map 
\begin{equation}\mathfrak{F}_{\varpi}\,:\, H\times \overline{U}\times B_{2\,\mathrm{d}}(\epsilon)\rightarrow M\,, \label{eq:F}\end{equation}
such that for any $(h,\, {m}_0)\in H\times \overline{U}$, the partial map 
\[ {\mathfrak{f}}^{(h,\,m_0)}_{\varpi} =:{\mathfrak{F}}_{\varpi}(h,\,m_0,\,\,\cdot,\,\cdot) \,:\, B_{2\mathrm{d}}(\epsilon)\times (-\pi,\,\pi) \rightarrow M\] is a preferred local chart for $M$ centered at ${\mu}_h(\sigma(m_0))$, whose image is a small neighborhood $U$ of ${\mu}_h(\sigma(m_0))$ in $M$. Furthermore, we adopt the following notation 
\begin{align}\left({\mathfrak{f}}^{(h,\,m_0)}_{\varpi}\right)^{-1}=:\left(z_1^{(h,m_0)},\dots,z_{\mathrm{d}}^{(h,\,m_0)} \right)\,:\,U\rightarrow B_{2\,\mathrm{d}}(\epsilon) \label{eq:modifiedf} \end{align}
for corresponding local coordinates. We identify $\mathbb{R}^{2\mathrm{d}}=\mathbb{C}^{\mathrm{d}}$ in the standard manner, thus $z^{(h,\,m_0)}_j\, :\, U \rightarrow \mathbb{C}$ in equation \eqref{eq:modifiedf} is a smooth function; define
 \begin{align} \label{eq:ab}
 a^{(h,\,m_0)}_j =:\Re(z^{(h,\,m_0)}_j),\qquad b^{(h,\,m_0)}_j =:\Im(z^{(h,\,m_0)}_j) \,:\, U \rightarrow \mathbb{R}\,.\end{align} 
 
It is natural here to use an equivariant version of \eqref{eq:heisenberg} when $x\in X_{\varpi}$, as introduced in \cite{pao-loa}, Section $2.9$. Let us define
\[S=:\exp_{\mathbb{T}}\left(\imath\mathbb{R}\cdot \varpi\right)\,. \]
For each $x\in X_{\varpi}$, we shall modify the map $\psi^{(x)}$ in order to incorporate the circle action $\tilde{\nu}\,:\,S \times X \rightarrow X$ (the restriction of  $\tilde{\mu}$ to $S$) instead of the standard one. For each $m\in M_{\varpi}$ we have  $\langle\Phi(m),\,\xi\rangle \neq 0$; then by \eqref{eqn:contact vector field} we have $\xi_X(x) \neq 0$ for every $x \in X_{\varpi}$. Hence $\tilde{\nu}$ is locally
free on $X_{\varpi}$, and every $x \in X_{\varpi}$ has finite stabilizer $S_x^1\subseteq S^1$. For each $x\in X_{\varpi}$, by composing with $\tilde{\nu}$, we obtain a map $(\psi_{\varpi}^{(x)})'\,:\,S^1\times B_{2\,\mathrm{d}}(\epsilon)\rightarrow X$ by letting
\begin{equation}(\psi_{\varpi}^{(x)})'(e^{\imath\,\vartheta},\,\mathbf{v}) =:\tilde{\nu}_{e^{\imath\vartheta}}(x+\mathbf{v})\,. \label{eq:newcoor} \end{equation} 
Working in coordinates on $S$, this yields a map $\psi_{\varpi}^{(x)}\,:\,(-\pi,\,\pi)\times B_{2\,\mathrm{d}}(\epsilon)\rightarrow X$ setting
\begin{equation}\psi_{\varpi}^{(x)}({\vartheta},\,\mathbf{v}) =:\tilde{\nu}_{e^{\imath\vartheta}}(x+\mathbf{v}) \qquad (x\in X_{\varpi})\,. \label{eq:newcoor1} \end{equation} 

The proof of the following Lemma is obtained arguing in a similar way as in Section $2.9$ of \cite{pao-fpf}, see especially Lemmas $39$ and $40$.

\begin{lem} \label{lem:varpicov}
	The map $(\psi_{\varpi}^{(x)})'$ is an $\lvert S_x^1\rvert:1$-covering, its image is an invariant tubular neighborhood of the $S^1$-orbit through $x$. 
\begin{proof}
	Let us denote with $\vartheta = (\vartheta_1,\dots,\,\vartheta_g)$ be the collective angular coordinates on $\mathbb{T}$, and set $\xi_j = \partial/\partial\vartheta_j\vert_{\boldsymbol{0}}$. Choose a basis of $\mathfrak{t}$ such that $\xi_1=:\varpi/\lVert\varpi\rVert$ and $\{\xi_2\,\dots \xi_g\}$ is a basis for $\ker(\Phi(m))$ for each $m\in M_{\varpi}$. Also, let
	\[\Phi_j =: \langle \Phi, \xi_j\rangle\,:\, M \rightarrow \mathbb{R},\quad \vartheta \cdot \Phi =:\sum^g_{j=1} \vartheta_j\cdot \Phi_j,\quad\text{and}\quad \vartheta\cdot \xi_M =:\sum^g_{j=1} \vartheta_j\,\xi_{j\,M}\,.\]
	
	First, let us prove that $\psi_{\varpi}^{(x)}$ is locally a diffeomorphism when $x\in X_{\varpi}$. Let us denote $m=\pi(x)$ and notice that $ \Phi_1(m)\neq 0$. By Corollary $2.2$ in \cite{pao-IJM}, we have
	\begin{align*} \tilde{\nu}_{-\vartheta}(x+\mathbf{v})=&x+\left(\vartheta\cdot \Phi_1(m)+\omega_m(\vartheta\cdot\xi_{1,\,M}(m),\,\mathbf{v}),\,\mathbf{v}-\vartheta\cdot\xi_{1,\,M}(m) \right)\\
	&+O(\lVert(\vartheta,\,\mathbf{v})\rVert^2)\,. \end{align*}
	The Jacobian matrix at the origin of $(\psi^{(x)})^{-1}\circ \psi^{(x)}_{\varpi}$ is then
	\begin{equation} \label{eq:jac}
	\mathrm{jac}_{(0,\,\mathbf{0})}\left( (\psi^{(x)})^{-1}\circ \psi^{(x)}_{\varpi}\right)=\begin{pmatrix}
	\mathrm{id}_{2\mathrm{d}} & \mathbf{H}(m) \\
	\boldsymbol{0}^t & \Phi_1(m)
	\end{pmatrix}\,, \end{equation}
	where $\mathbf{H}(m)\in \mathbb{R}^{2\mathrm{d}}$ is the local coordinate expression of $\xi_{1,\,M}(m) \in T_mM$ (viewed as a column vector). Thus $\psi_{\varpi}$ is a local diffeomorphism at $(1, \boldsymbol{0}) \in \mathbf{T}^1 \times B_{2\,\mathrm{d}}(\boldsymbol{0}, \epsilon)$. Therefore, if $T\subseteq \mathbf{T}^1$ is a sufficiently small neighborhood of the identity, then for small $\epsilon >0$, the restriction of $\psi_{\varpi}$ to $\mathbf{T}^1 \times B_{2\mathrm{d}}(\boldsymbol{0}, \epsilon)$ is a diffeomorphism onto its image. 
	
	Since the stabilizer of $x$ is $S^1_x$, the inverse image $(\psi_{\varpi}^{(x)})'^{-1}(y)$ contains at least $\lvert S_x^1\rvert$ distinct elements. Arguing as in Lemma 39 of \cite{pao-loa} one can actually prove that it has exactly $\lvert S_x^1\rvert$ inverse images, sufficiently close to the orbit through $x$.
\end{proof}
\end{lem}

Recall that, for each $l=1,\,\dots\,\ell$, $F_{\overline{X}_{\varpi},\,l}$ is the connected component of $\mathrm{Fix}(\gamma_{\,\overline{X}_{\varpi}})$ and $\widetilde{F}_{\overline{X}_{\varpi},\,l}$ is the pullback of the locus $F_{\overline{X}_{\varpi},\,l}$ on $X_{\varpi}$. Given $x\in \widetilde{F}_{\overline{X}_{\varpi},\,l}$ with $\pi(x)=m$, then for each $\mathbf{v}\in T_mM_{\varpi}$ the map $\psi_{\varpi}^{(x)}({\vartheta},\,\mathbf{v}) $ defines local coordinates at $x$ in an open neighborhood $U$ in $X_{\varpi}$. 

Let $F_{{\overline{M}_{\varpi}},\,l}$ and $N_{{\overline{M}_{\varpi}},\,l}$ be as in the Introduction. We denote with $N_{M,\,l}$ the normal bundle in $M$ to $p_{{\overline{M}_{\varpi}}}^*(F_{{\overline{M}_{\varpi}},\,l})$. By the $\mathbb{T}$-invariance of $M_{\varpi}$ and $F_{M,\,l}$, we can modify the local coordinates of equations \eqref{eq:modifiedf} and \eqref{eq:ab}. Hence for every $h\in H$ and $m_0 \in \overline{U}$ we obtain preferred local coordinates centered in $\mu_{\sigma}(h,\,m_0)$ such that the following two conditions are satisfied
\begin{itemize}	\item[i)] $M_{\varpi}\cap U=\left\{b^{(h,\,m_0)}_{\mathrm{d-g+2}}=\dots=b^{(h,\,m_0)}_{\mathrm{d}}=0 \right\}$;	\item[ii)] $F_{M,\,l} \cap U=\left\{z^{(h,\,m_0)}_{\mathrm{d}_l+1}=\dots=z^{(h,\,m_0)}_{\mathrm{d}-g+1}=b^{(h,\,m_0)}_{\mathrm{d-g+2}}=\dots=b^{(h,\,m_0)}_{\mathrm{d}}=0 \right\}$.\end{itemize}

Along the line of \cite{pao-IJM}, Section \S $2.2$, we introduce a decomposition of $T_{m}M$ when $m\in M_{\varpi}$. Under the transversality assumption, we have
\begin{equation} T_mM= T^{\mathrm{t}}_mM\oplus T^{\mathrm{v}}_mM \oplus T^{\mathrm{h}}_mM\,, \label{eq:decomposition} \end{equation}
where
\[T^{\mathrm{t}}_m M:=J_m\,\varpi^{\perp_{\mathfrak{t}}}_M(m)\,, \qquad T^{\mathrm{v}}_mM:=\varpi^{\perp_{\mathfrak{t}}}_M(m)\,\]
and
\[T^{\mathrm{h}}_mM := \left(T^{\mathrm{v}}_mM \oplus T^{\mathrm{t}}_mM \right)^{\perp_{g}}\,. \]
We can use this decomposition to write $\mathbf{v}$ as
\[\mathbf{v}= \mathbf{v}_{\mathrm{h},\,\mathrm{nor}}+\mathbf{v}_{\mathrm{t}} \]
where $\mathbf{v}_{\mathrm{h},\,\mathrm{nor}}\in T^{\mathrm{h}}_mM\cap {N}_{M,\,l}(m)$ and $\mathbf{v}_{\mathrm{t}}\in T^{\mathrm{t}}_mM$. In such a way that, for each $j=0\,\dots,\,\mathrm{g-2}$ and $i=0\,\dots,\,\mathrm{c}_l-1$, we can write
$$b^{(h,\,m_0)}_{\mathrm{d}-j}=\mathbf{v}_{\mathrm{t}}  \qquad\text{ and }\qquad z^{(h,\,m_0)}_{\mathrm{d}-g+1-i}=\mathbf{v}_{\mathrm{h},\,\mathrm{nor}}\,.$$

\begin{lem} \label{lem:changeofcoor} Suppose that $f$ is a continuous function on $X$. Consider a small equivariant tubular neighborhood $\mathcal{T}(F_{X,\,l})$ of $p^*_{\overline{X}_{\varpi}}(F_{\overline{X}_{\varpi},\,l})$ in $X$, we have
\begin{align*}\int_{\mathcal{T}(F_{X,\,l})} f \,\mathrm{dV}_X =& \int_{-\pi}^{\pi} \lvert\mathrm{d}\vartheta\rvert\, \int_H\,\mathrm{dV}_H(h)\int_{B_{2\mathrm{d}_l}(\epsilon)}\mathrm{dV}_{F_{X,\,l}}(x_0)\, \int_{B_{2\mathrm{c}_l+g-1}(\epsilon)}\mathrm{d}\mathcal{L}(\mathbf{v})\\
&  \frac{1}{2\,\pi\,\lvert \mathbf{T}\rvert}\cdot\left[ f\left(\psi_{\varpi}^{(x_0,h)}(\vartheta,\,\mathbf{v})\right)\cdot \mathcal{D}(m) \cdot (\lVert\Phi(m)\rVert+A(\mathbf{v}))  \right] \end{align*}
where $A(\mathbf{v}) = O(\mathbf{v})$ and
\[\mathcal{D}(m)=: \sqrt{D(m)} \qquad (m \in M_{\varpi})\,; \]
where $D(m)$ is the change of basis matrix between the two Euclidean structures on $\ker(\Phi(m))$ induced from $\mathfrak{t}$ and $T_mM$, respectively. 

\end{lem} \begin{proof}[Proof of Lemma \ref{lem:changeofcoor}]

As a consequence of \eqref{eq:jac} and Lemma \ref{lem:varpicov}, we have 
\begin{align} \label{eq:pullbackvolume}
&(\psi_{\varpi}^{(x)})^*(\mathrm{dV}_X) = (\psi^{(x)}\circ (\psi^{(x)})^{-1}\circ \psi_{\varpi}^{(x)})^*(\mathrm{dV}_X) \\
&\quad=((\psi^{(x)})^{-1}\circ \psi_{\varpi}^{(x)})^*((\psi^{(x)})^*\mathrm{dV}_X) \notag \\
&\quad= ((\psi^{(x)})^{-1}\circ \psi_{\varpi}^{(x)})^*
( \mathcal{V}(\vartheta, \mathbf{v}) \lvert\mathrm{d}\vartheta\rvert \mathrm{d}\mathcal{L}(\mathbf{v})) \notag \\
&\quad=\left(\mathcal{V}\circ \left(((\psi^{(x)})^{-1}\circ \psi_{\varpi}^{(x)})\right)\right)\cdot \rvert\det\left(\mathrm{jac}\left(((\psi^{(x)})^{-1}\circ \psi_{\varpi}^{(x)})\right)\right)\rvert
\, \lvert\mathrm{d}\vartheta\rvert\mathrm{d}\mathcal{L}(\mathbf{v})\,.\notag
\end{align}
Inserting the determinant of \eqref{eq:jac} in the last line of \eqref{eq:pullbackvolume}, we get at $(0,\, \boldsymbol{0})$
\begin{equation} \label{eq:pullbackat00}
(\psi_{\varpi}^{(x)})^*(\mathrm{dV}_X)(0,\, \boldsymbol{0}) =
\frac{1}{2\pi}
\Phi_1(m)\, |\mathrm{d}\vartheta| \mathrm{d}\mathcal{L}(\mathbf{v}).\end{equation}
To prove \eqref{eq:pullbackat00} at $(\vartheta_0, 0)$, we replace $\vartheta \simeq \vartheta_0$ by $\vartheta + \vartheta_0$ with $\vartheta \simeq 0$ and note that $\nu_{\vartheta + \vartheta_0}(x + \mathbf{v}) = \nu_{\vartheta}\left(\nu_{\vartheta_0}(x + \mathbf{v})\right)$. Since
\[\psi^{(x)}({\vartheta_0},\,\mathbf{v})=e^{\imath\,\theta}\cdot \left(\nu_{\vartheta_0}(x+\mathbf{v}) \right) \]
is a system of Heisenberg Local Coordinates centered at $\nu_{\vartheta_0}(x)$, one can argue as in the previous case. Also notice that for each $m\in M_{\varpi}$ we have $\Phi(m)=(\lVert\Phi(m)\rVert\varpi)/\lVert\varpi\rVert$ and thus $\Phi_1(m)=\lVert\Phi(m)\rVert$.

The underlying preferred coordinates on $M$ can also be modified by including the action of $H$ as in \eqref{eq:F}. In view of Lemma \ref{lem:musigma} the map $\mu_{\sigma}$ is a $\lvert \mathbf{H} \rvert$-to-1 covering of the orbit $H\cdot m$. Since the map $h\mapsto h\cdot m$ is a covering map of the $H$-orbit, we have 
\[V_{\mathrm{eff}}^H(m)\cdot \lvert H_m\rvert= \int_{-\pi}^{\pi} \mathrm{d}\vartheta_1\cdots \int_{-{\pi}}^{\pi}\mathrm{d}\vartheta_{g-1}\left[\sqrt{D(m)} \right]= ({2\,\pi})^{\mathrm{g}-1}\,\sqrt{D(m)}\,, \]
where $V_{\mathrm{eff}}^H(m)$ is the volume of the orbit $H\cdot m$ and $D(m)$ is defined in the statement of Lemma. By invoking Lemma $3.9$ of \cite{DP} we can conclude.
\end{proof}

Let us now return to \eqref{eq:tracevarrho}. For every $l=1,\,\dots, \ell$, let $\{U_{l\,h} \}_h$ be a finite open cover of $p^*_{\overline{X}_{\varpi}}(F_{\overline{X}_{\varpi},\,l})$ and let
\begin{equation} \label{eq:tauij}
\{\tau_{l\,h}\,: U_{l\,h} \rightarrow \mathbb{R} \,\}_{l,\,h}
\end{equation}
be an equivariant smooth partition of unity subordinate to this cover. By making use of Lemma \ref{lem:changeofcoor} and inserting the relation \eqref{eq:tauij} in \eqref{eq:tracevarrho} we obtain the following oscillatory integrals
\[\mathrm{trace}\left(\Psi_{k\varpi} \right) \sim   \sum_{l\,h}\mathrm{trace}\left(\Psi_{k\varpi} \right)_{l\,h} \] where 
\begin{align}& \label{eq:tracelj} \mathrm{trace}\left(\Psi_{k\varpi} \right)_{l\,h} =\int_{\mathbb{R}^{\mathrm{c}_l}}\mathrm{d}\mathcal{L}(\mathbf{v}_{\mathrm{h},\,\mathrm{nor}})\,\int_{\mathbb{R}^{\mathrm{g}-1}}\mathrm{d}\mathcal{L}(\mathbf{v}_{\mathrm{t}}) \int_{\mathbb{T}}\,\mathrm{dV}_{\mathbb{T}}(t)\,\int_{U_{l\,h}}\mathrm{dV}_{F_{\overline{X}_{\varpi},\,l}}(x_0) \\
&\biggl[\varrho_{\mathrm{h},\mathrm{nor}}\left(k^{1/2-\epsilon}\,\mathbf{v}_{\mathrm{h},\mathrm{nor}}\right)\,\varrho_{\mathrm{t}}\left(k^{1/2-\epsilon}\,\mathbf{v}_{\mathrm{t}}\right)\,\tau_{l\,h}(x_0)\,f(\tilde{\mu}_t(\sigma(x_0)+\mathbf{v}_{\mathrm{h},\,\mathrm{nor}}+\mathbf{v}_{\mathrm{t}}))\, \notag \\ &\quad \cdot   \frac{1}{\lvert \mathbf{T}\rvert}\cdot \mathcal{D}(m)\cdot (\lVert\Phi(m)\rVert+A(\mathbf{v}))  \notag \\ & \quad \cdot \Pi_{k\varpi} \left(\gamma^{-1}_X(\tilde{\mu}_h(\sigma(x_0))+\mathbf{v}_{\mathrm{h},\,\mathrm{nor}}+\mathbf{v}_{\mathrm{t}}),\,\tilde{\mu}_h(\sigma(x_0))+\mathbf{v}_{\mathrm{h},\,\mathrm{nor}}+\mathbf{v}_{\mathrm{t}}\right)\biggr]\,; \notag\end{align} 
noticed that in \eqref{eq:tracelj} we use that the action of $\mathbb{T}$ commutes with $\gamma_X$ and
\begin{equation} \label{eq:Pikvarpich}
\Pi_{k\varpi}(x,\,t\cdot y)=\overline{\chi_{k\varpi}({t})}\cdot \Pi_{k\varpi}(x,\,y)\,.\end{equation}

Now we shall use the asymptotic expansion for scaling limits of equivariant Szeg\"o kernels obtained in \cite{pao-IJM}. Theorem $4$ in \cite{pao-IJM} concerns off-locus asymptotics expansions along directions in the normal bundle $T^{\mathrm{t}}M$. One can generalize them for arbitrary displacements as follow.\footnote{These expansions come from unpunished notes of Prof. R. Paoletti. Although in a more general setting they are also contained in \cite{cm}.} For each $x\in X_{\varpi}$ and $v_i\in T_xX$, $i=1,\,2$ and $v_i=(\theta_i,\mathbf{v}_i)$, we have
\begin{align} \label{eq:piknu}
&\Pi_{k\varpi}\left(x+\frac{v_{1}}{\sqrt{k}},\,x+\frac{v_2}{\sqrt{k}} \right)  \\
&\qquad\sim  \frac{e^{\imath\,\sqrt{k}\,\lambda_{\varpi}(m)\,(\theta_1-\theta_2)}}{(\sqrt{2}\,\pi)^{g-1}\,\mathcal{D}(m)}\,\left(\lVert \varpi \rVert\cdot\frac{k}{\pi}  \right)^{\mathrm{d}+(1-\mathrm{g})/2}\cdot\left( \frac{1}{\lVert \Phi(m) \rVert}  \right)^{\mathrm{d}+1+(1-\mathrm{g})/2} \notag\\
&\qquad\,\cdot  \left(\sum_{t\in T_m}\chi_{\varpi}(t)^k\,e^{\lambda_{\varpi}(m)\,E_m\left(\mathrm{d}_m\tilde{\mu}_{t^{-1}}(\mathbf{v}),\mathbf{w}\right)}\right) \cdot \left(1+\sum_{j\geq 1}R_j(m,\,v_1,\,v_2)\,k^{-j/2} \right)\,, \notag
\end{align}
where $\lambda_{\varpi}(m)=:\lVert \varpi \rVert / \lVert \Phi(m) \rVert$, $R_j$ is polynomial in the $v_i$'s,
\begin{align} \label{eq:E}
E(v_1,\,v_2):=\, & \imath\,\omega_m\left(\mathbf{A}(v_1,\,v_2),,\mathbf{v}_{1\,\mathrm{h}} \right) +\psi_2\left(\mathbf{v}_{1\,\mathrm{h}}-\mathbf{A}(v_1,\,v_2)_{\mathrm{h}},\,\mathbf{v}_{2\,\mathrm{h}} \right) \\
& -\lVert \mathbf{v}_{1\,\mathrm{t}} \rVert^2-\lVert \mathbf{v}_{2\,\mathrm{t}}\rVert^2 \notag \\
& \imath\,\omega_m\left(\mathbf{v}_{1\,\mathrm{v}},\,\mathbf{v}_{1\,\mathrm{t}}\right) -\imath\,\omega_m\left(\mathbf{v}_{2\,\mathrm{v}},\,\mathbf{v}_{2\,\mathrm{t}}\right)\,, \notag
\end{align}
and 
\begin{equation} \label{eq:A}\mathbf{A}(v_1,\,v_2)=:\frac{\theta_2-\theta_1}{\lVert \Phi(m) \rVert}\,\xi_{1,\,M}\,. \end{equation}

We need to adapt the fourth line of equation \eqref{eq:tracelj} to make use of the expansion \eqref{eq:piknu}. Pick $x_0\in F_{X,\,l}$, for each $j=1,\,\dots,\,\mathbf{T}$ there exists $\kappa_j\in \mathbb{T}$ such that 
$$\gamma_X^{-1}(\sigma(x_0))=\tilde{\mu}_{\kappa_j}(\sigma(x_0))\,.$$ 
Fix $\kappa=\kappa_j$, of course $\kappa_j=\kappa\cdot t_j$ with $t_j\in \mathbf{T}$. Let us denote with $x(\kappa,\,h)=\tilde{\mu}_{\kappa\,h}(\sigma(x_0))$, $x(h)=\tilde{\mu}_{h}(\sigma(x_0))$ and similarly $m(h)={\mu}_{h}(\sigma(m_0))$. Furthermore, since the map $\gamma_X\,:\,X\rightarrow X$ is a Riemannian isometry, we deduce that $$\gamma^{-1}_X\circ \exp_x=\exp_{\gamma^{-1}_X(x)}\circ\, \mathrm{d}_x\gamma^{-1}_X\,.$$
Hence, for all $\mathbf{v}\in T_{\mu_h(\sigma(m_0))} M$ we have
\begin{equation} \label{eq: gammaMcoordinates} 
\gamma^{-1}_X\left(x(h)+\mathbf{v}\right)=x(\kappa,\,h)+\mathrm{d}_{x(h)}\gamma_X^{-1}(\mathbf{v})\,. \end{equation}
Similarly, we should write
\begin{equation} \label{eq: Picoordinatesright} 
x(h)+\mathbf{v}=\tilde{\mu}_{\kappa^{-1}}\left(x(\kappa,\,h)+\mathrm{d}_{x(h)}\tilde{\mu}_{\kappa}(\mathbf{v})\right)\,.  \end{equation}
Thus, inserting equations \eqref{eq: gammaMcoordinates}  and \eqref{eq: Picoordinatesright} in the last line of \eqref{eq:tracelj}, we obtain
\begin{align*}
&\Pi_{k\varpi} \left(\gamma^{-1}_X(x(h)+\mathbf{v}),\,x(h)+\mathbf{v}\right) \\
&\quad =\Pi_{k\varpi} \left(x(\kappa,\,h)+\mathrm{d}_{x(h)}\gamma_X^{-1}(\mathbf{v}),\,\tilde{\mu}_{\kappa^{-1}}\left(x(\kappa,\,h)+\mathrm{d}_{x(h)}\tilde{\mu}_{\kappa}(\mathbf{v})\right)\right)
\\
&\quad =\overline{\chi_{k\varpi}({\kappa})}\cdot \Pi_{k\varpi} \left(x(\kappa,\,h)+\mathrm{d}_{x(h)}\gamma_X^{-1}(\mathbf{v}),\,x(\kappa,\,h)+\mathrm{d}_{x(h)}\tilde{\mu}_{\kappa}(\mathbf{v})\right)
\end{align*}
where we applied \eqref{eq: gammaMcoordinates}, \eqref{eq: Picoordinatesright} and \eqref{eq:Pikvarpich}.

Let us set $\mathbf{v}=\mathbf{v}_{\mathrm{h},\,\mathrm{nor}}+\mathbf{v}_{\mathrm{t}}$. The map $\mathrm{d}\gamma_X^{-1}$ preserves the norm and decomposition defined in \eqref{eq:decomposition}, thus we have
\begin{equation} \label{eq:gammapreservesdecomp} \mathrm{d}_{x(h)}\gamma_X^{-1}(\mathbf{v}_{\mathrm{h},\,\mathrm{nor}}+\mathbf{v}_{\mathrm{t}})= \mathrm{d}_{x(h)}\gamma_X^{-1}(\mathbf{v}_{\mathrm{h},\,\mathrm{nor}})+ \mathrm{d}_{x(h)}\gamma_X^{-1}(\mathbf{v}_{\mathrm{t}}) \end{equation}
and where the former (respectively, the latter) term in the right hand side lies in $T_{m(h)}^{\mathrm{h}}M\cap N_{M,\,l}$ (respectively in $T_{m(h)}^{\mathrm{t}}M$). The analogue of \eqref{eq:gammapreservesdecomp} for $\mathrm{d}_{x(h)}\tilde{\mu}_{\kappa}$ holds true and furthermore $\mathrm{d}_{x(h)}\gamma_X^{-1}$ and $\mathrm{d}_{x(h)}\tilde{\mu}_{\kappa}$ preserve the norm. Thus, the expression \eqref{eq:E} has the following form
\begin{align*}
E(\mathrm{d}_{x(h)}\gamma_X^{-1}(\mathbf{v}),\,\mathrm{d}_{x(h)}\tilde{\mu}_{\kappa}(\mathbf{v})):=\, & \psi_2\left(\mathrm{d}_{x(h)}\gamma_X^{-1}(\mathbf{v}_{\mathrm{h},\,\mathrm{nor}}),\,\mathrm{d}_{x(h)}\tilde{\mu}_{\kappa}(\mathbf{v}_{\mathrm{h},\,\mathrm{nor}}) \right) \\
&-2\,\lVert \mathbf{v}_{\mathrm{t}} \rVert^2.  
\end{align*}

We rescale $\mathbf{v}$ by a factor $k^{-1/2}$ and integrate the rescaled variables over a ball of radius $\sim k^{\epsilon}$ in \eqref{eq:tracelj}, this gives a factor $k^{-c_l-(g-1)/2}$. Using Taylor expansions, we obtain 
\begin{align} \label{eq:f}
f\left(\tilde{\nu}_{e^{\imath\,\theta}}\left(x(h)+\frac{\mathbf{v}_{\mathrm{h},\,\mathrm{nor}}}{\sqrt{k}}+\frac{\mathbf{v}_{\mathrm{t}}}{\sqrt{k}}\right)\right)&= f(x(g))+\sum_{j\geq 1} k^{-j/2}\,f_j(\mathbf{v})\,,
\end{align}
and similarly, adapting \eqref{eq:piknu} in this case, we have
\begin{align} \label{eq:piknuadapt}
&\Pi_{k\varpi}\left(x(\kappa,\,h)+\mathrm{d}_{x(h)}\gamma_X^{-1}\left(\frac{\mathbf{v}}{\sqrt{k}}\right),\,x(\kappa,\,h)+\mathrm{d}_{x(h)}\tilde{\mu}_{\kappa}\left(\frac{\mathbf{v}}{\sqrt{k}}\right)\right)\,  \\
&\qquad\sim  \frac{1}{(\sqrt{2}\,\pi)^{g-1}\,\mathcal{D}(m) }\,\left(\lVert \varpi \rVert\cdot\frac{k}{\pi}  \right)^{\mathrm{d}+(1-\mathrm{g})/2}\cdot\left( \frac{1}{\lVert \Phi(m) \rVert}  \right)^{\mathrm{d}+(1-\mathrm{g})/2+1} \notag\\
&\qquad\,\cdot  \left(\sum_{t\in \mathbf{T}}\overline{\chi_{k\varpi}(t)}\cdot e^{\lambda_{\varpi}(m)\,\left[\psi_2\left(\mathrm{d}_{x(h)}\gamma_M^{-1}\left(\mathrm{d}_m\tilde{\mu}_{t}(\mathbf{v}_{\mathrm{h},\,\mathrm{nor}})\right),\, \mathrm{d}_{x(h)}\tilde{\mu}_{\kappa}\left(\mathbf{v}_{\mathrm{h},\,\mathrm{nor}}\right) \right) -2\,\lVert \mathbf{v}_{\mathrm{t}} \rVert^2\right]}\right) \notag \\
&\qquad\, \cdot \left(1+\sum_{j\geq 1}R_j(m,\,\mathbf{v})\,k^{-j/2} \right)\,. \notag
\end{align}

Every vector $\mathbf{v}\in T_m^{\mathrm{h}}M_{\varpi}\setminus\{\boldsymbol{0}\}$ descends naturally to a non vanishing vector in $T_{m_{0}}\overline{M}_{\varpi}$ (where $m_0=p_{M,\,\varpi}(m)$):
\begin{align*} 
 T_m^{\mathrm{h}}M_{\varpi} &\rightarrow T_{m_0}\overline{M}_{\varpi} \\
\mathbf{v}  &\mapsto \mathbf{v}_{\varpi}\,.  
\end{align*}
Thus $\psi_2$ descends to a function on $T_m\overline{M}_{\varpi}$. Furthermore, by Definition \ref{def:cl}, for ever $t\in \mathbb{T}_m$ let us denote with $\Lambda_l^{(t)}\in U(c_l)$ the restriction of $\mathrm{d}\gamma_{\overline{M}_{\varpi}}\circ\mathrm{d}\mu_{t^{-1}}$ to the normal subspace of $F_{\overline{M}_{\varpi},\,l}$ at $m_0$ in the induced coordinates. Let $\{\mathbf{v}_1,\dots\mathbf{v}_{c_l}\}$ be an orthonormal basis of $\mathbb{C}^{c_l}$ composed of eigenvectors of $\Lambda_l^{(t)}$, with the corresponding eigenvalues $\lambda_1^{(t)}\dots\lambda_{c_{l}}^{(t)}\in S^{1}\setminus \{1\}$. If $(\mathbf{v}_{\mathrm{h},\mathrm{nor}})_{\varpi}=\sum_{j=1}^{c_l}a_j\,\mathbf{v}_j$, we have
\begin{equation} \label{eq:psi2lambda}
\psi_2\left(  \left(\Lambda_l^{(t)}\right)^{-1}(\mathbf{v}_{\mathrm{h},\,\mathrm{nor}} )_{\varpi},\,(\mathbf{v}_{\mathrm{h},\,\mathrm{nor}})_{\varpi}\right)=\sum_{j=1}^{\mathrm{c}_l}(\bar{\lambda}_j^{(t)}-1)\,\lvert a_j \rvert^2\,.
\end{equation}
Furthermore, notice that
\begin{equation} \label{eq:change}
\lambda_{\varpi}(m)\,\psi_2\left( \mathbf{u} ,\,\mathbf{w}\right) = \,\psi_2\left( \sqrt{\lambda_{\varpi}(m)}\,\mathbf{u} ,\,\sqrt{\lambda_{\varpi}(m)}\,\mathbf{w}\right)= \psi_2\left( \mathbf{u}' ,\,\mathbf{w}'\right) \,,
\end{equation}
where we set $\mathbf{v}'=\sqrt{\lambda_{\varpi}(m)}\,\mathbf{v}$.

Arguing in a similar way as in \cite{pao-tf} (see especially equations $(63)$ and $(64)$) and making use of \eqref{eq:psi2lambda} and \eqref{eq:change}, we obtain
\begin{align} \label{eq:secondint}
&\sum_{t\in \mathbf{T}} \chi_{\varpi}(t)^k\,\int_{\mathbb{C}^{c_l}}\exp\left({\lambda_{\varpi}(m)\,\psi_2\left(  \left(\Lambda_l^{(t)}\right)^{-1}(\mathbf{v}_{\mathrm{h},\,\mathrm{nor}} )_{\varpi},\,(\mathbf{v}_{\mathrm{h},\,\mathrm{nor}})_{\varpi}\right) }\right) \mathrm{d}\mathcal{L}(\mathbf{v}_t) \\
&= \lambda_{\varpi}(m)^{c_l}\cdot\sum_{t\in \mathbf{T}} \chi_{\varpi}(t)^k\,\int_{\mathbb{C}^{c_l}}\exp\left({\sum_{j=1}^{\mathrm{c}_l}(\bar{\lambda}_j^{(t)}-1)\,\lvert a_j \rvert^2 }\right) \mathrm{d}a \notag \\&= \lambda_{\varpi}(m)^{c_l}\cdot\sum_{t\in \mathbf{T}} \chi_{\varpi}(t)^k\,\prod_{j=1}^{\mathrm{c}_l}\int_{\mathbb{C}^{c_l}}\exp\left({(\bar{\lambda}_j^{(t)}-1)\,\lvert u \rvert^2 }\right) \mathrm{d}u \notag \\&= \lambda_{\varpi}(m)^{c_l}\cdot\sum_{t\in \mathbf{T}} \chi_{\varpi}(t)^k\,\prod_{j=1}^{\mathrm{c}_l}\frac{1}{\bar{\lambda}_j^{(t)}-1 }  \notag \\
& = \lambda_{\varpi}(m)^{c_l}  \cdot \sum_{t\in \mathbf{T}} \chi_{\varpi}(t)^k\,\frac{\pi^{\mathrm{c}_l}}{c_l^{(t)}(\gamma)}\,. \notag
\end{align}

Furthermore, we are led to compute the following integral
\begin{align} \label{eq:firstint}
\int_{\mathbb{R}^{g-1}}e^{-2\,\lambda_{\varpi}(m)\,\lVert \mathbf{v}_t\rVert^2}\, \mathrm{d}\mathcal{L}(\mathbf{v}_t)\,=  \left(\frac{\pi}{2\,\lambda_{\varpi}(m)}\right)^{(g-1)/2}\,.
\end{align}
To complete the proof of Theorem \ref{thm:main} we need only to insert \eqref{eq:piknuadapt} and \eqref{eq:f} in \eqref{eq:tracelj}, making use of \eqref{eq:firstint} and \eqref{eq:secondint}. Eventually, notice that the integral over $F_{\overline{X}_{\varpi},l}$ can be written as an integral over $F_{\overline{M}_{\varpi},l}$ by simply pulling back the measure $\mathrm{dV}_{{\overline{X}_{\varpi}}}$ on $\mathrm{dV}_{X}$ and then pushing forward on $F_{\overline{M}_{\varpi},l}$:
\[ p_{\overline{M}_{\varpi}}^*(\mathrm{dV}_{F_{\overline{M}_{\varpi},l}})=\lVert \Phi(m) \rVert\,p_{\overline{X}_{\varpi}}^*(\mathrm{dV}_{F_{\overline{M}_{\varpi},l}})\,.  \]

\end{document}